\documentclass[11pt]{amsart}

\usepackage[a4paper,hmargin=3.5cm,vmargin=3.9cm]{geometry}
\usepackage{amsfonts,amssymb,amscd,amstext}
\usepackage{graphicx}
\usepackage[dvips]{epsfig}
\usepackage{quoting}
\usepackage{hyperref}

\usepackage{fancyhdr}
\input xy
\xyoption{all}

\usepackage{enumerate}
\usepackage{titlesec}
\usepackage{mathrsfs}

\titleformat{\section}
{\filcenter\bfseries\large} {\thesection{.}}{0.2cm}{}

\titleformat{\subsection}[runin]
{\bfseries} {\thesubsection{.}}{0.15cm}{}[.]

\titleformat{\subsubsection}[runin]
{\em}{\thesubsubsection{.}}{0.15cm}{}[.]

\usepackage[up,bf]{caption}

\newtheorem{theorem}{Theorem}[section]

\newtheorem{lemma}[theorem]{Lemma}
\newtheorem{claim}[theorem]{Claim}
\newtheorem{corollary}[theorem]{Corollary}

\theoremstyle{definition}

\newtheorem{remark}[theorem]{Remark}

\numberwithin{equation}{section}
\numberwithin{figure}{section}

\newcommand\Cscr{\mathscr{C}}

\newcommand\Oscr{\mathscr{O}}

\newcommand\Gscr{\mathscr{G}}

\newcommand\C{\mathbb{C}}

\renewcommand\P{\mathbb{P}}

\newcommand\R{\mathbb{R}}

\newcommand\Z{\mathbb{Z}}

\newcommand\igot{\mathfrak{i}}

\renewcommand\igot{\mathfrak{i}}

\renewcommand\imath{\igot}

\newcommand\Flux{\mathrm{Flux}}

\newcommand\CMI{\mathrm{CMI}}

\def\Flux{\mathrm{Flux}}

\newcommand\CMIf{\mathrm{CMI_{full}}}
\newcommand\CMIfc{\mathrm{CMI_{full}^c}}
\newcommand\CMInf{\mathrm{CMI_{nf}}}

\newcommand\RNCf{\Re\mathrm{NC_{full}}}
\newcommand\RNCfc{\Re\mathrm{NC_{full}^c}}

\newcommand\RNC{\Re\mathrm{NC}}

\newcommand\Of{\mathscr O_{\mathrm{full}}}
\newcommand\Ofc{\mathscr O_{\mathrm{full}}^\mathrm{c}}

\newcommand\Onc{\mathscr O_{\mathrm{nc}}}
\newcommand\Oncc{\mathscr O_{\mathrm{nc}}^\mathrm{c}}
\newcommand\boldA{\mathbf A}

\usepackage{color}

\begin{document}


\fancyhead[LO]{The space of Gauss maps of complete minimal surfaces}
\fancyhead[RE]{A.\ Alarc\'on and F.\ L\'arusson}
\fancyhead[RO,LE]{\thepage}

\thispagestyle{empty}



\begin{center}
{\bf \LARGE The space of Gauss maps
\\ \vspace{1mm}  
of complete minimal surfaces} 

\bigskip

%
%
{\large\bf Antonio Alarc\'on \; and \; Finnur L\'arusson}
\end{center}


%
%

\bigskip

\begin{quoting}[leftmargin={5mm}]
{\small
\noindent {\bf Abstract}\hspace*{0.1cm}  
The Gauss map of a conformal minimal immersion of an open Riemann surface $M$ into $\R^3$ is a meromorphic function on $M$.  In this paper, we prove that the Gauss map assignment, taking a full conformal minimal immersion $M\to\R^3$ to its Gauss map, is a Serre fibration. We then determine the homotopy type of the space of meromorphic functions on  $M$ that are the Gauss map of a complete full conformal minimal immersion, and show that it is the same as the homotopy type of the space of all continuous maps from $M$ to the 2-sphere.  
We obtain analogous results for the generalised Gauss map of conformal minimal immersions $M\to\R^n$ for arbitrary $n\ge 3$.
 

\noindent{\bf Keywords}\hspace*{0.1cm} 
Riemann surface, minimal surface, complete minimal surface, Gauss map, h-principle, Oka manifold, Serre fibration, weak homotopy equivalence


\noindent{\bf Mathematics Subject Classification (2020)}\hspace*{0.1cm} 
Primary 53A10. Secondary 30F99, 32E30, 32H02, 32Q56

\noindent{\bf Date}\hspace*{0.1cm} 
22 May 2022
}

\end{quoting}



\section{Introduction and main results}
\label{sec:intro}

\noindent
The Gauss map of a minimal surface in $\R^3$, parametrised as a conformal minimal immersion from an open Riemann surface $M$ into $\R^3$, may be viewed as a meromorphic function on $M$.  Bonnet first observed this fact in 1860 \cite{Bonnet1860} and Christoffel proved in 1867 \cite{Christoffel1867} that it characterises minimal surfaces in $\R^3$.  Via the Gauss map, complex-analytic methods have ever since played a major role in the classical theory of minimal surfaces.  The literature is vast.  We refer to \cite[Chapter 12]{Osserman1986} and \cite[Chapter 5]{AlarconForstnericLopez2021} for historical background and further references.  

It is a long-standing unsolved problem in the global theory of minimal surfaces to usefully characterise those meromorphic functions that are the Gauss map of a {\em complete\,} minimal surface.  Several decades of research on Picard-type theorems for Gauss maps of complete minimal surfaces culminated in the 1988 theorem of Fujimoto that such a map can omit at most four values in the Riemann sphere unless the surface is a plane \cite{Fujimoto1988JMSJ}.  This result is sharp.  Some further restrictions were given by Osserman \cite{Osserman1963} and by Weitsman and Xavier \cite{WeitsmanXavier1987}.  As an example in the other direction, Su and Li recently produced a sufficient Nevanlinna-theoretic condition for a meromorphic function on the plane or the disc to be the Gauss map of a complete minimal surface \cite{SuLi2019, SuLi2019b}.

In this paper, we take a new approach to the problem.  We investigate the space of meromorphic functions on $M$ that are the Gauss map of a complete minimal surface from a homotopy-theoretic viewpoint.  We determine the homotopy type of this space.  One of our main results is that the inclusion of this space in the space of all meromorphic functions on $M$ is a weak homotopy equivalence and, when $M$ has finite topological type, even a genuine homotopy equivalence.

It was discovered only recently that every meromorphic function on $M$ is the Gauss map of a conformal minimal immersion $M\to\R^3$ \cite{AlarconForstnericLopez2019JGEA}.  We show that the Gauss map assignment is not only surjective: it is in fact a Serre fibration.  This is a key ingredient in the proof of the main result described above, along with the strong parametric h-principle for complete minimal surfaces in our previous paper~\cite{AlarconLarusson2021}.

Our results extend to all higher dimensions.  To more precisely present them, we need to introduce some notation.  Let $M$ be an open Riemann surface, throughout assumed connected, and let $n\ge 3$. If $u=(u_1,\ldots,u_n): M \to \R^n$ is a conformal minimal immersion, then the $(1,0)$-differential $\partial u$ of $u$ determines the Kodaira-type holomorphic map $\Gscr(u)$ from $M$ into the hyperquadric 
\[ \bold Q^{n-2} = \big\{ [z_1:\cdots: z_n] \in \C\P^{n-1}\colon z_1^2+\cdots+z_n^2=0 \big\} \]
in $\C\P^{n-1}$ given by
\[
	\Gscr(u)(p) = [\partial u_1(p) : \cdots : \partial u_n(p)],\quad p\in M.
\]
The map $\Gscr(u)$ is called the {\em generalised Gauss map} of $u$, or, in this paper, simply the {\em Gauss map} of $u$.  A conformal immersion $M\to\R^n$ is minimal if and only if its Gauss map is holomorphic \cite[Theorem 1.1]{HoffmanOsserman1980}.\footnote{The Gauss map defined in \cite{HoffmanOsserman1980} is the conjugate of the Gauss map defined here.}

By \cite[Theorem 1.1]{AlarconForstnericLopez2019JGEA}, the Gauss map assignment $\Gscr : \CMI(M,\R^n) \to \Oscr(M,\bold Q^{n-2})$ is surjective.  Here, $\CMI(M,\R^n)$ and $\Oscr(M,\bold Q^{n-2})$ denote the spaces of all conformal minimal immersions $M\to\R^n$ and of all holomorphic maps $M\to \bold Q^{n-2}$, respectively.  The key to this and other recent applications of Oka theory in the theory of minimal surfaces is the fact that $\bold Q^{n-2}$ is an Oka manifold (see \cite{AlarconForstneric2014IM} or \cite[Example 5.6.2]{Forstneric2017E}).

A holomorphic map $M\to\C\P^{n-1}$ is said to be {\em full\,} if its image is not contained in any hyperplane; we denote by $\Of(M,\bold Q^{n-2})$ the open subspace of $\Oscr(M,\bold Q^{n-2})$ consisting of full maps.  A conformal minimal immersion $u\in \CMI(M,\R^n)$ is called {\em full\,} if its Gauss map $\Gscr(u)$ is full, and we denote by $\CMIf(M,\R^n)$ the open subspace of $\CMI(M,\R^n)$ consisting of all such immersions.  We endow these spaces with the compact-open topology.

The {\em flux\,} $\Flux(u)$ of an immersion $u\in\CMI(M,\R^n)$ is the cohomology class of its conjugate differential $d^cu=\imath(\bar\partial u-\partial u)$ in $H^1(M,\R^n)$.  The flux is naturally identified with the group homomorphism $\Flux(u):H_1(M,\Z)\to\R^n$ given by
\[
	\Flux(u)([C])=\int_C d^cu = -2\imath\int_C \partial u,\quad [C]\in H_1(M,\Z).
\]
We view the cohomology group $H^1(M,\C^n)$ as the de Rham group of $n$-tuples of holomorphic $1$-forms on $M$ modulo exact forms, endowed with the quotient topology induced from the compact-open topology.  The subgroup $H^1(M,\R^n)$ carries the subspace topology.

The first main result of this paper states that the Gauss map assignment for full conformal minimal immersions $\Gscr: \CMIf(M,\R^n) \to \Of(M,\bold Q^{n-2})$ is a Serre fibration, that is, satisfies the homotopy lifting property with respect to all CW-complexes. In fact, we prove the following stronger result.
%
%
\begin{theorem}\label{th:fibration}
If $M$ is an open Riemann surface and $n\ge 3$, then the map
\[
	(\Gscr,\Flux) : \CMIf(M,\R^n) \to \Of(M,\bold Q^{n-2})\times H^1(M,\R^n)
\]
is a Serre fibration.
\end{theorem}

The fact that the flux map $\CMIf(M,\R^n) \to H^1(M,\R^n)$ is a Serre fibration was already known (\cite[Theorem 6.1(a)]{AlarconLarusson2021} is stated for complete immersions, but in its proof completeness may be ignored).  We prove Theorem \ref{th:fibration} in Section \ref{sec:proofs} as a consequence of the main technical result of the paper, Theorem \ref{th:main-technical-theorem}, which is stated and proved in Section~\ref{sec:technical}.  Our proofs rely on the maps we are working with being full.  The key applications of fullness, in the proofs of Lemma \ref{lem:critical-case} and Corollary \ref{co:R-full-k}, have been highlighted for the reader's convenience.

Our second main result is a contribution to the open problem of determining which holomorphic maps $M\to \bold Q^{n-2}$ are Gauss maps of complete conformal minimal immersions.  As already mentioned for $n=3$, the study of the value distribution properties of the Gauss map of complete minimal surfaces in $\R^n$ for $n\ge 3$ has been one of the main foci of interest in this theory.  Some restrictions are known.  Ru proved that the Gaussian image of a complete nonflat minimal surface in $\R^n$ cannot omit more than $n(n+1)/2$ hyperplanes in $\C\P^{n-1}$ in general position \cite{Ru1991JDG}.  This is sharp whenever $n$ is odd or at most $17$ \cite{Fujimoto1988SRKU}. The same result for full minimal surfaces was previously obtained by Fujimoto \cite{Fujimoto1983JMSJ, Fujimoto1990JDG}.  Let $\CMIfc(M, \R^n)$ denote the subspace of $\CMIf(M,\R^n)$ of complete conformal minimal immersions.  It follows from the parametric h-principle that is the main result of our paper \cite{AlarconLarusson2021} that the inclusion $\CMIfc(M, \R^n) \hookrightarrow \CMIf(M,\R^n)$ is a weak homotopy equivalence with dense image.  Let $\Ofc(M,\bold Q^{n-2}) = \Gscr(\CMIfc(M, \R^n))$.

\begin{theorem}\label{th:inclusion}
Let $M$ be an open Riemann surface and $n\ge 3$.

{\rm (a)}  The inclusion $\Ofc(M,\bold Q^{n-2}) \hookrightarrow \Of(M,\bold Q^{n-2})$ is a weak homotopy equi\-valence.

{\rm (b)}  If $M$ has finite topological type,\footnote{We recall the following equivalent definitions of finite topological type:  the fundamental group of $M$ is finitely generated; $M$ has the homotopy type of a finite bouquet of circles; $M$ can be obtained from a compact Riemann surface by removing a finite number of mutually disjoint points and closed discs; $M$ has a strictly subharmonic Morse exhaustion with finitely many critical points.} then the inclusion is a homotopy equivalence.

{\rm (c)}  The inclusion $\Ofc(M,\bold Q^{n-2}) \hookrightarrow \Cscr(M,\bold Q^{n-2})$ is a weak homotopy equivalence, and, if $M$ has finite topological type, a homotopy equivalence.
\end{theorem}

Part {\rm (a)} means that the inclusion induces a bijection of path components $\pi_0(\Ofc(M,\bold Q^{n-2}))\to \pi_0(\Of(M,\bold Q^{n-2}))$ and an isomorphism of homotopy groups
\[
	\pi_k(\Ofc(M,\bold Q^{n-2}),g)\to \pi_k(\Of(M,\bold Q^{n-2}),g)
\]
for every integer $k\ge 1$ and every base point $g=\Gscr(u)\in \Ofc(M,\bold Q^{n-2})$, with $u\in \CMIfc(M, \R^n)$.  By (b), when $M$ is of finite topological type, there is a homotopy inverse $\xi:\Of(M,\bold Q^{n-2}) \to \Ofc(M,\bold Q^{n-2})$ to the inclusion.  This means that there is a way to associate to every map $g\in \Of(M,\bold Q^{n-2})$ a map $\xi(g)\in \Ofc(M,\bold Q^{n-2})$ that is homotopic to $g$.  Moreover, if $g\in\Ofc(M,\bold Q^{n-2})$ to begin with, then there is such a homotopy through maps in $\Ofc(M,\bold Q^{n-2})$.  The main point is that $\xi(g)$ and the homotopies depend continuously on $g$.  Finally, part (c) reduces the determination of the homotopy type of $\Ofc(M,\bold Q^{n-2})$ to a purely topological problem.

The main ingredients in the proof of the theorem, which is given in Section \ref{sec:proofs}, are the parametric h-principle in \cite{AlarconLarusson2021}, the result that the Gauss map assignment is a fibration (Theorem \ref{th:fibration}), and, for part (b), the theory of absolute neighbourhood retracts (ANRs) in the category of metric spaces and \cite[Theorem 9]{Larusson2015PAMS}, which uses Oka theory to show that certain spaces of holomorphic maps are ANRs.

We observe that since the space $\CMIfc(M, \R^n)$ is dense in $\CMIf(M,\R^n)$ (\cite[Theorem 7.1]{AlarconForstnericLopez2016MZ}; the case of $n=3$ follows from \cite[Theorem 5.6]{AlarconLopez2012JDG}) and the map $\Gscr:\CMIf(M, \R^n) \to \Of(M,\bold Q^{n-2})$ is surjective \cite[Theorem 1.1]{AlarconForstnericLopez2019JGEA}, $\Ofc(M,\bold Q^{n-2})$ is dense in $\Of(M,\bold Q^{n-2})$.  As a consequence of Theorem \ref{th:fibration}, $\Ofc(M,\bold Q^{n-2})$ is dense in $\Of(M,\bold Q^{n-2})$ in the following stronger sense. See \cite[Corollary 1.3]{AlarconLarusson2021} for an analogous result for the subspace $\CMIfc(M, \R^n)$ of $\CMIf(M,\R^n)$.
 
\begin{corollary}\label{co:dense}
If $M$ is an open Riemann surface, $P$ is a contractible finite CW-complex, and $Q\subset P$ is a retract of $P$, then every continuous map $Q\to \Of(M,\bold Q^{n-2})$, $n\ge 3$, extends to a continuous map $P\to \Of(M,\bold Q^{n-2})$ that takes $P\setminus Q$ into $\Ofc(M,\bold Q^{n-2})$.
\end{corollary}

\begin{remark} Consider the commuting square
\[ \xymatrix{
	\CMIfc(M, \R^n)  \ar@{^{(}->}[r]^i \ar[d]^{\Gscr^c}  &  \CMIf(M,\R^n) \ar[d]^\Gscr \\ 
	\Ofc(M,\bold Q^{n-2})   \ar@{^{(}->}[r]^j     &  \Of(M,\bold Q^{n-2})  } \]
where $\Gscr^c$ is the restriction of $\Gscr$ to $\CMIfc(M, \R^n)$, that is, the Gauss map assignment for full complete conformal minimal immersions. We know that the inclusions $i$ and $j$ are weak homotopy equivalences by \cite[Theorem 6.1(c)]{AlarconLarusson2021} and Theorem \ref{th:inclusion}(a), respectively, while $\Gscr$ is a fibration by Theorem \ref{th:fibration}. It remains an open question whether $\Gscr^c$ is a fibration as well.
\end{remark}

In the classical case of $n=3$, a conformal minimal immersion $u$ is full if and only if it is nonflat, that is, its image does not lie in an affine 2-plane in $\R^3$.  Equivalently, the Gauss map of $u$ is not constant.  Also, $\bold Q^1$ may be identified with the Riemann sphere $\P=\C\P^1$ and the Gauss map of a conformal minimal immersion $u=(u_1,u_2,u_3):M\to\R^3$ viewed, via the stereographic projection, as the holomorphic function $M\to\P$ given by
\[
	\Gscr(u)=\frac{\partial u_3}{\partial u_1-\imath \partial u_2},
\]
often called the {\em complex Gauss map} of $u$ (see \cite[Section 2.5]{AlarconForstnericLopez2021} for more details).  In the following corollary, the subscript nf stands for nonflat and nc for nonconstant.

\begin{corollary}\label{cor:dim-three}
Let $M$ be an open Riemann surface.

{\rm(a)}  $(\Gscr,\Flux) : \CMInf(M,\R^3) \to \Onc(M,\P)\times H^1(M,\R^3)$ is a Serre fibration.

{\rm(b)}  The inclusion $\Oncc(M,\P) \hookrightarrow \Onc(M,\P)$ is a weak homotopy equivalence, whose image is dense in the strong sense of Corollary \ref{co:dense}.  The inclusion $\Oncc(M,\P) \hookrightarrow \Cscr(M,\P)$ is a weak homotopy equivalence.  If $M$ has finite topological type, then the inclusions are homotopy equivalences.
\end{corollary}

\begin{remark}\label{rem:null-curves}
Let us briefly indicate how our results can be adapted to null curves.  A conformal minimal immersion $u:M\to\R^n$ has a harmonic conjugate $v$ if and only if its flux vanishes.  Then the holomorphic map $u+\imath v:M\to\C^n$ is a so-called null curve.  The subspace of $\CMI(M,\R^n)$ of immersions with vanishing flux, that is, real parts of holomorphic null curves, is denoted $\RNC(M,\C^n)$.  By \cite[Theorem 1.1]{AlarconForstnericLopez2019JGEA}, the Gauss map assignment $\Gscr:\RNCf(M,\C^n) \to \Of(M, \bold Q^{n-2})$ is surjective, and Theorem \ref{th:fibration} implies that it is a fibration (where the subscripts have the usual meaning).  

The proof of Theorem \ref{th:inclusion}(a) is then easily adapted, using the control on the flux provided by \cite[Theorem 6.1(a)]{AlarconLarusson2021}, to show that the inclusion into $\Of(M, \bold Q^{n-2})$ of the space $\Gscr(\RNCfc(M, \R^n))$ of full holomorphic maps $M\to\bold Q^{n-2}$ that are the Gauss map of the real part of a complete holomorphic null curve is a weak homotopy equivalence.  Moreover, the inclusion has dense image (using the analogue for full immersions of \cite[Corollary 1.3]{AlarconLarusson2021}, which is an immediate consequence of \cite[Theorem 6.1(a)]{AlarconLarusson2021}).  We also see that the inclusion $\Gscr(\RNCfc(M, \R^n)) \hookrightarrow \Ofc(M, \bold Q^{n-2})$ is a weak homotopy equivalence.  It is an open question whether the two spaces are in fact the same.

More generally, if we fix $\alpha\in H^1(M, \R^n)$, the corresponding results hold for conformal minimal immersions with flux $\alpha$.
\end{remark}

Theorem \ref{th:fibration} implies that the space $\Gscr^{-1}(g)$ of full conformal minimal immersions $M\to \R^n$ with fixed Gauss map $g$ has the same weak homotopy type for all $g\in\Of(M, \bold Q^{n-2})$.  In Section \ref{sec:fibre} we determine this homotopy type.

\begin{theorem}\label{th:fibre}
Let $M$ be an open Riemann surface and $n\ge 3$.  The fibre of the Gauss map assignment $\Gscr: \CMIf(M,\R^n) \to \Of(M,\bold Q^{n-2})$ has the weak homo\-topy type of a countably infinite disjoint union of circles, unless $M$ is the plane or the disc, in which case the fibre has the weak homotopy type of a circle.
\end{theorem}

Whether $\Gscr : \CMI(M,\R^n) \to \Oscr(M,\bold Q^{n-2})$ is a fibration or whether there are singularities of some sort over the non-full maps that prevent $\Gscr$ from being a fibration is an open question that is beyond our techniques at present.  For $n=3$, the fibre of $\Gscr$ over a constant map essentially consists of the holomorphic immersions $M\to\C$.  As determined in \cite{ForstnericLarusson2019CAG}, the space of such maps has the weak homotopy type of $\Cscr(M,\C^*)$.  As shown in Section \ref{sec:fibre}, the fibre of $\Gscr$ over a full immersion has that same homotopy type, suggesting that the question might have an affirmative answer.


\section{The main technical theorem}
\label{sec:technical}

\noindent
According to \cite[Definition 1.12.9]{AlarconForstnericLopez2021}, a compact subset $S\neq\varnothing$ of an open Riemann surface $M$ is {\em admissible} if it is $\Oscr(M)$-convex and of the form $S=K\cup\Gamma$, where $K$ is the union of finitely many pairwise disjoint smoothly bounded compact domains in $M$ and $\Gamma= S \setminus K$ is a finite union of pairwise disjoint smooth Jordan arcs meeting $K$ only at their endpoints (or not at all) and such that their intersections with the boundary of $K$ are transverse.
Given such a set $S=K\cup\Gamma$ and a complex submanifold $Z\subset\C^n$, we denote by $\mathscr{A}(S,Z)$ the space of all continuous maps $S\to Z$ that are holomorphic on $\mathring S=\mathring K$. For simplicity, we write $\mathscr{A}(S)=\mathscr{A}(S,\C)$.

In this section we prove the following theorem, which is the technical heart of the paper.
%
%
\begin{theorem}\label{th:main-technical-theorem}
Let $M$ be an open Riemann surface, $\theta$ be a holomorphic $1$-form vanishing nowhere on $M$, $S=K\cup\Gamma\subset M$ be an admissible subset, $k$ and $n$ be positive integers, and $f_p\colon M\to\C^n$ and $F_p\in H^1(M,\C^n)$ $(p\in [0,1]^k)$ be continuous families of full holomorphic maps and cohomology classes. Then, every continuous family of functions $\varphi_p:S\to\C^*=\C\setminus\{0\}$ $(p\in [0,1]^k)$ of class $\mathscr{A}(S)$ satisfying
\begin{equation}\label{eq:main-technical-theorem}
	\int_C \varphi_pf_p\theta = F_p([C])\quad 
	\text{for all closed curves $C\subset S$ and all $p\in [0,1]^k$}
\end{equation}
can be approximated uniformly on $[0,1]^k\times S$ by continuous families of holomorphic functions $\widetilde \varphi_p:M\to\C^*$ $(p\in [0,1]^k)$ such that
\[
	\int_C \widetilde\varphi_pf_p\theta = F_p([C])\quad 
	\text{for all closed curves $C\subset M$ and all $p\in [0,1]^k$}.
\]
Furthermore, if $\varphi_p$ is holomorphic on $M$, vanishes nowhere on $M$, and satisfies 
\[
	\int_C \varphi_pf_p\theta=F_p([C])\quad 
	\text{for all closed curves $C\subset M$, for all $p\in [0,1]^{k-1}\times\{0\}$},
\]
then we can choose $\widetilde \varphi_p=\varphi_p$ for all $p\in [0,1]^{k-1}\times\{0\}$.
\end{theorem}

The case $k=1$ of Theorem \ref{th:main-technical-theorem} was proved in \cite[Theorem 4.1]{AlarconForstnericLopez2019JGEA}; the proof relies in an essential way on the parameter space $[0,1]$ being 1-dimensional (see \cite[proof of Lemma 2.3]{AlarconForstnericLopez2019JGEA}).  The proof of Theorem \ref{th:main-technical-theorem} follows the scheme of the proof of \cite[Theorem 4.1]{AlarconForstnericLopez2019JGEA}, but with an additional idea that enables us to work with a parameter space of arbitrary dimension. In particular, we shall make use of \cite[Lemmas 3.2 and 4.2]{AlarconForstnericLopez2019JGEA}. The main new technical ingredient in our proof is the following extension of \cite[Lemma 4.3]{AlarconForstnericLopez2019JGEA} to the parameter space $[0,1]^k$ for arbitrary $k\ge 1$.

%
%
\begin{lemma}[The critical case]\label{lem:critical-case}
Let $M$, $\theta$, $k$, $n$, and $f_p$ $(p\in [0,1]^k)$ be as in Theorem \ref{th:main-technical-theorem}.   Also let $\rho:M\to [0,+\infty)$ be a smooth strongly subharmonic Morse exhaustion function and let $0<a<b$ be a pair of regular values of $\rho$ such that $\rho$ has a single critical point in $L\setminus\mathring K$, where $K=\{\rho\le a\}$ and $L=\{\rho\le b\}$. Assume that we have continuous families of functions $\varphi_p:K\to\C^*$ of class $\mathscr{A}(K)$ and cohomology classes $F_p\in H^1(M,\C)$ $(p\in [0,1]^k)$ such that 
\[
	\int_C \varphi_pf_p\theta = F_p([C])\quad
	\text{for all closed curves $C\subset K$ and all }p\in[0,1]^k.
\]
Then, the family $\varphi_p$ can be approximated uniformly on $[0,1]^k\times K$ by continuous families of functions $\widetilde\varphi_p:L\to\C^*$ $(p\in [0,1]^k)$ of class $\mathscr{A}(L)$ such that
\[
	\int_C \widetilde\varphi_pf_p\theta = F_p([C])\quad
	\text{for all closed curves $C\subset L$ and all }p\in[0,1]^k.
\]
Furthermore, if $\varphi_p$ is of class $\mathscr{A}(L)$, vanishes nowhere on $L$, and satisfies
\[
	\int_C \varphi_pf_p\theta = F_p([C])\quad
	\text{for all closed curves $C\subset L$, for all }p\in[0,1]^{k-1}\times\{0\},
\]
then we can choose $\widetilde\varphi_p=\varphi_p$ for all $p\in [0,1]^{k-1}\times\{0\}$.
\end{lemma}
We note that \cite[Lemma 4.2]{AlarconForstnericLopez2019JGEA} holds in our more general framework with the same proof but replacing the parameter space $[0,1]$ by $[0,1]^k$ for the given integer $k\ge 1$. We record this result for later reference.
%
%
\begin{lemma}[The noncritical case]\label{lem:noncritical-case}
Let $M$, $S=K\cup\Gamma$, $\theta$, $k$, $n$, and $f_p$ $(p\in [0,1]^k)$ be as in Theorem \ref{th:main-technical-theorem}. If $L\subset M$ is a smoothly bounded $\Oscr(M)$-convex compact domain such that $S\subset\mathring L$ and $S$ is a deformation retract of $L$, then every continuous family of functions $\varphi_p:S\to\C^*$ $(p\in [0,1]^k)$ of class $\mathscr{A}(S)$ can be uniformly approximated on $[0,1]^k\times S$ by continuous families of functions $\widetilde \varphi_p:L\to\C^*$ $(p\in [0,1]^k)$ of class $\mathscr{A}(L)$ such that
$(\widetilde\varphi_p-\varphi_p)f_p\theta$ is exact on $S$ for all $p\in [0,1]^k$.
	
Furthermore, if $\varphi_p$ is of class $\mathscr{A}(L)$ and vanishes nowhere on $L$  for all $p\in [0,1]^{k-1}\times\{0\}$, then we can choose $\widetilde\varphi_p=\varphi_p$ for all $p\in [0,1]^{k-1}\times\{0\}$.
\end{lemma}

%
%
\begin{proof}[Proof of Theorem \ref{th:main-technical-theorem} assuming Lemma \ref{lem:critical-case}]
Choose a smooth strongly subharmonic Morse exhaustion function $\rho:M\to\R$ and a divergent sequence of regular values $0<a_1<a_2<\cdots$ of $\rho$ such that, setting $K_0=S$ and $K_j=\{\rho\le a_j\}$ for all $j\ge 1$, the following conditions are satisfied.
\begin{itemize}
\item $K_0\subset\mathring K_1$ and $K_0$ is a strong deformation retract of $K_1$.
\smallskip
\item $\rho$ has at most a single critical point in $K_{j+1}\setminus \mathring K_j$ (which lies in $\mathring K_{j+1}\setminus K_j$) for all $j\ge 1$.
\end{itemize}
It turns out that $K_j$ is a smoothly bounded $\Oscr (M)$-convex compact domain for all $j\ge 1$, and
\[
	S=K_0\Subset K_1\Subset K_2\cdots\Subset \bigcup_{j\ge 0} K_j=M
\]
is an exhaustion of $M$. Let $\varphi_p:S\to\C^*$ $(p\in [0,1]^k)$ be a continuous family of functions of class $\mathscr{A}(S)$ satisfying condition \eqref{eq:main-technical-theorem}, and fix $\epsilon>0$. Call $\varphi_p^0=\varphi_p$ for all $p\in [0,1]^k$. A standard recursive application of Lemmas \ref{lem:critical-case} and \ref{lem:noncritical-case} provides a sequence of continuous families of functions $\varphi_p^j:K_j\to\C^*$ $(p\in [0,1]^k)$, $j\ge 1$, of class $\mathscr{A}(K_j)$ satisfying the following conditions for all $j\ge 1$.
\begin{itemize}
\item $\varphi_p^j$ is as close as desired to $\varphi_p^{j-1}$ uniformly on $[0,1]^k\times K_{j-1}$.
\smallskip
\item $\displaystyle\int_C \varphi_p^j f_p\theta = F_p([C])$ for all closed curves $C\subset K_j$ and all $p\in[0,1]^k$.
\smallskip
\item If $\varphi_p$ is holomorphic on $M$, vanishes nowhere on $M$, and satisfies 
\[
	\int_C \varphi_pf_p\theta=F_p([C])\quad 
	\text{for all closed curves $C\subset M$, for all $p\in [0,1]^{k-1}\times\{0\}$},
\]
then we can choose $\varphi_p^j=\varphi_p^0=\varphi_p$ for all $p\in [0,1]^{k-1}\times\{0\}$.
\end{itemize}
If we take $\varphi_p^j$ sufficiently close to $\varphi_p^{j-1}$ on $[0,1]^k\times K_{j-1}$ at each step, we obtain in the limit a continuous family of holomorphic functions 
\[
	\widetilde \varphi_p:=\lim_{j\to\infty}\varphi_p^j:M\to\C^*\quad (p\in [0,1]^k)
\] 
which is $\epsilon$-close to the family $\varphi_p$ on $[0,1]^k\times S$ and satisfies the requirements in the theorem.
\end{proof}

To complete the proof of Theorem \ref{th:main-technical-theorem} it remains to prove Lemma \ref{lem:critical-case}. We begin with some preparations. Given an integer $n\ge 1$, we shall say that a continuous map $f:[0,1]\to\R^n$ is {\em $\R$-full\,} if its image is contained in no real linear hyperplane, that is, the real span of $f([0,1])$ equals $\R^n$. Likewise, a continuous map $f:[0,1]\to\C^n$ is said to be {\em $\C$-full\,} if the complex span of $f([0,1])$ equals $\C^n$. It is clear that every $\R$-full map $[0,1]\to\C^n=\R^{2n}$ is $\C$-full, but the converse does not hold true in general.
%
%
\begin{lemma}\label{lem:R-full-k}
Let $k$ and $n$ be positive integers, $P=[0,1]^k$, and $f:P\times [0,1]\to\R^n$ and $\alpha:P\to\R^n$ be continuous maps. If the path $f_p:=f(p,\cdot):[0,1]\to\R^n$ is $\R$-full for every $p\in P$, then for any $\epsilon>0$ there exists a continuous function $x:P\times [0,1]\to\R$ such that $x(p,s)=0$ for $p\in P$ and $s\in\{0,1\}$ and
\[
	\left|\int_0^1x(p,s)f(p,s)\, ds - \alpha(p)\right| < \epsilon \quad \text{for all }p\in P.
\]
If in addition $\alpha(p)=0$ for all $p\in Q=[0,1]^{k-1}\times\{0\}\subset P$, then we can choose $x$ with $x(p,s)=0$ for all $p\in Q$ and $s\in [0,1]$.
\end{lemma} 
In case $k=1$, we identify $Q$ with $\{0\}\subset P=[0,1]$.
%
%
In the proof we shall use the following observation, which corresponds to the case $k=0$ in the lemma.
\begin{claim}\label{cl:k=0}
If $f:[0,1]\to \R^n$ $(n\ge 1)$ is continuous and $\R$-full, then for any $\alpha\in\R^n$ and any $\epsilon>0$ there exists a continuous function $x:[0,1]\to\R$ such that $x(0)=x(1)=0$ and
\[
	\left| \int_0^1 x(s) f(s)\, ds - \alpha \right| < \epsilon.
\]
\end{claim}

\begin{proof}
If $\alpha=0\in\R^n$, then we simply choose $x=0$. Assume that $\alpha\neq 0$.  By $\R$-fullness of $f$ there are points $0<s_1<s_2<\cdots<s_n<1$ in $[0,1]$ such that
\begin{equation}\label{eq:spanR-k=0}
	{\rm span}_\R\{f(s_1),\ldots,f(s_n)\}=\R^n.
\end{equation}
Thus, there is a number $\sigma>0$ so small that the intervals $[s_j-\sigma,s_j+\sigma]$, $j\in\{1,\ldots,n\}$, lie in $(0,1)$ and are pairwise disjoint. Set 
\[
	v_j=\int_{s_j-\sigma}^{s_j+\sigma} f(s)\, ds\in\R^n,\quad j=1,\ldots,n,
\]
and note that $v_j$ is close to $2\sigma f(s_j)$ provided that $\sigma$ is small.
We assume as we may by \eqref{eq:spanR-k=0} that $\sigma>0$ is so small that
${\rm span}_\R\{v_1,\ldots,v_n\}=\R^n$,
and write 
\[
	\alpha=\sum_{j=1}^n\lambda_jv_j
\] 
for (unique) $\lambda_1,\ldots,\lambda_n\in\R$. The step function $x_0:[0,1]\to\R$ given by
\[
	x_0(s)=\left\{
	\begin{array}{ll}
	\lambda_j & \text{ if }s\in [s_j-\sigma,s_j+\sigma],\; j=1,\ldots,n, \medskip
	\\
	0 & \text{ if }s\in [0,1]\setminus\bigcup_{j=1}^n [s_j-\sigma,s_j+\sigma],
	\end{array}\right.
\]
satisfies $x_0(0)=x_0(1)=0$ and $\int_0^1 x_0(s)f(s)=\alpha$. Suitably deforming $x_0$ in a small neighbourhood of the points $s_j\pm\sigma$, $j=1,\ldots,n$, to make it continuous, we obtain a function $x:[0,1]\to\R$ satisfying the required conditions.
\end{proof}

%
%
\begin{proof}[Proof of Lemma \ref{lem:R-full-k}]
We proceed by induction on the positive integer $k$. For the base case when $k=1$, we have $Q=\{0\}\subset P=[0,1]$ and a pair of continuous maps $f:P\times[0,1]\to\R^n$ and $\alpha:P\to\R^n$. Since the map $f_p:[0,1]\to\R^n$ is $\R$-full for every $p\in P$, Claim \ref{cl:k=0} gives a continuous function $y_p:[0,1]\to\R$ such that \begin{equation}\label{eq:yp(0)}
	y_p(0)=y_p(1)=0
\end{equation}
and
\[
	\left| \int_0^1 y_p(s) f_p(s)\, ds - \alpha(p) \right| < \epsilon\quad \text{for all }p\in P.
\]
If $\alpha(0)=0$ then we choose $y_0=0$. The problem now is that $y_p$ does not depend continuously on $p\in P$, so we have to do some more work.
By continuity of $f$ and $\alpha$ and compactness of $P$, there is a partition $0=p_0<p_1<\cdots<p_k=1$ of $P=[0,1]$ such that 
\begin{equation}\label{eq:y_p_j}
	\left| \int_0^1 y_{p_j}(s) f_p(s)\, ds - \alpha(p) \right| < \epsilon
	\quad \text{for all }p\in [p_{j-1},p_j],\; j=1,\ldots,k.
\end{equation}
The function $x:P\times[0,1]\to\R$ given by
\[
	x(p,\cdot)=\frac{p_j-p}{{p_j-p_{j-1}}}\, y_{p_{j-1}}+\frac{p-p_{j-1}}{p_j-p_{j-1}}\, y_{p_j}
	\quad \text{for all }p\in [p_{j-1},p_j],\; j=1,\ldots,k,
\]
is continuous and, in view of \eqref{eq:yp(0)}, satisfies $x(p,0)=x(p,1)=0$ for all $p\in P$. Moreover, \eqref{eq:y_p_j} ensures that
\[
	\left| \int_0^1 x(p,s) f(p,s)\, ds - \alpha(p) \right| < \epsilon\quad \text{for all }p\in P.
\]
Finally, note that $x(0,\cdot)=x(p_0,\cdot)=y_0$; hence $x(0,s)=0$ for all $s\in [0,1]$ provided that $\alpha(0)=0$. This proves the base case.

For the inductive step, fix an integer $k\ge 2$ and assume that the lemma holds for $P'=[0,1]^{k-1}$ and $Q'=[0,1]^{k-2}\times\{0\}$. If $k=2$, we identify $Q'$ with $\{0\}\subset P'=[0,1]$. We write $P=[0,1]^k=P'\times[0,1]$ and $Q=[0,1]^{k-1}\times\{0\}=P'\times\{0\}$. We use the same argument as above. Fix $t\in[0,1]$. The map $f_q^t:=f_{(q,t)}:[0,1]\to\R^n$ is $\R$-full for every $q\in P'$, and hence the induction hypothesis provides a continuous function $y_t:P'\times [0,1]\to\R$ such that $y_t(q,0)=y_t(q,1)=0$ for all $q\in P'$,
and
\[
	\left| \int_0^1 y_t(q,s) f^t_q(s)\, ds - \alpha(q,t) \right| < \epsilon\quad \text{for all }q\in P'.
\]
If $\alpha(q,0)=0$ for all $q\in P'$ then we choose $y_0=0$. Again, the problem is that the map $y_t$ does not depend continuously on $t\in[0,1]$.
To arrange this, take a partition $0=t_0<t_1<\cdots<t_k=1$ of $[0,1]$ such that 
\[
	\left| \int_0^1 y_{t_j}(q,s) f_q^t(s)\, ds - \alpha(q,t) \right| < \epsilon
	\quad \text{for all }q\in P'\text{ and }t\in [t_{j-1},t_j],\; j=1,\ldots,k.
\]
The function $x:P\times[0,1]=P'\times[0,1]\times[0,1]\to\R$ given by
\[
	x(\cdot,t,\cdot)=\frac{t_j-t}{{t_j-t_{j-1}}}\, y_{t_{j-1}}+\frac{t-t_{j-1}}{t_j-t_{j-1}}\, y_{t_j}
	\quad \text{for all }t\in [t_{j-1},t_j],\; j=1,\ldots,k,
\]
satisfies the required conditions. This completes the induction.
\end{proof}

%
%
\begin{corollary}\label{co:R-full-k}
Let $k$, $n$, $P$, and $Q$ be as in Lemma \ref{lem:R-full-k}, let $f:P\times [0,1]\to\C^n$ and $\alpha:P\to\C^n$ be continuous maps, and assume that there is a pair of closed intervals $J$ and $J'$ such that $J\cap J'=\varnothing$, $J\cup J'\subset (0,1)$, and the path $f_p:=f(p,\cdot):[0,1]\to\C^n$ is $\C$-full on $J$ and $\R$-full on $J'$ for every $p\in P$. Then, there exists a continuous function $h:P\times [0,1]\to\C^*$ such that $h(p,0)=h(p,1)=1$ for all $p\in P$ and
\[
	\int_0^1 h(p,s)f(p,s)\, ds = \alpha(p) \quad \text{for all }p\in P.
\]
If in addition $\alpha(p)=\int_0^1 f(p,s)\, ds$ for all $p\in Q\subset P$, then there is such a function $h$ with $h(p,s)=1$ for all $p\in Q$ and $s\in [0,1]$.

Furthermore, the function $h$ can be chosen such that $|\Re(h)-1|<\sigma$ in $P\times[0,1]$ for any given number $\sigma>0$; in particular, we can choose $h$ so that $\Re(h)>0$.
\end{corollary} 

The final claim of the corollary is not needed here, but is included for possible future applications.

%
%
\begin{proof}
Consider the period map $\mathscr{P}:\mathscr{C}([0,1],\C^n)\to\C^n$ given by
\[
	\mathscr{P}(g)=\int_0^1 g(s)\, ds\in\C^n,\quad g\in \mathscr{C}([0,1],\C^n).
\]
By \cite[Lemma 2.1]{AlarconForstnericLopez2019JGEA}, there are continuous functions $g_1,\ldots,g_N:[0,1]\to\C$ $(N\ge n)$, supported on $J$, such that the function $w:\C^N\times[0,1]\to\C$ given by
\[
	w(\zeta,s):=\prod_{i=1}^N \big(1+\zeta_i g_i(s)\big),
	\quad \zeta=(\zeta_1,\ldots,\zeta_N)\in\C^N,\; s\in [0,1],
\]
has
\[
	\frac{\partial}{\partial\zeta} \mathscr{P}\big(w(\zeta,\cdot) f(p,\cdot)\big)\big|_{\zeta=0}: T_0\C^N\cong\C^N\to\C^n
	\quad\text{surjective for every }p\in P.
\]
(Here we need the maps $f_p$ to be full.) In particular,
\begin{equation}\label{eq:w(p,)}
	w(\cdot,s)=1\quad \text{ for all } s\in [0,1]\setminus J\supset\{0,1\}.
\end{equation}
 By virtue of the implicit function theorem, this implies that for every neighbourhood $U$ of $0$ in $\C^N$, there is a number $\epsilon>0$ with the following property: if $\beta: P\to\C^n$ is a continuous map with
$| \mathscr{P}(f(p,\cdot)) - \beta(p) | <\epsilon$ for all $p\in P$, then there is a continuous map $\zeta_\beta:P\to U$ such that
\[
	\mathscr{P}\big(w(\zeta_\beta(p),\cdot) f(p,\cdot)\big) = \beta(p)\quad \text{for all }p\in P.
\]
Furthermore, if $\mathscr{P}(f(p,\cdot)) = \beta(p)$ for all $p\in Q$, then we can choose the map $\zeta_\beta$ such that $\zeta_\beta(p)=0$ (and hence $w(\zeta_\beta(p),\cdot)=1$) for all $p\in Q$.
Fix a number $0<\sigma<1$, let $U$ be a neighbourhood of $0\in\C^N$ so small that 
\begin{equation}\label{eq:Rew-1}
	|\Re(w)-1|<\sigma\quad \text{in }U\times[0,1]
\end{equation}
(hence $\Re(w)>0$ there), and let $\epsilon>0$ be a number satisfying the above condition.

Consider the continuous map $\gamma: P\to\C^n$ given by 
\begin{equation}\label{eq:gamma(p)}
	\gamma(p)=-\imath\big(\alpha(p)-\mathscr{P}(f(p,\cdot))\big),\quad p\in P.
\end{equation}
Lemma \ref{lem:R-full-k} furnishes us with a continuous function $x:P\times [0,1]\to\R$, supported on $P\times J'$, such that  
\[
	\big| \mathscr{P}\big(x(p,\cdot)f(p,\cdot)\big) - \gamma(p) \big| < \epsilon \quad \text{for all }p\in P,
\]
and if $\gamma(p)=0$ for all $p\in Q$, then $x(p,\cdot)=0$ for all $p\in Q$. It turns out that the continuous function $\tilde h=1+\imath x\colon P\times[0,1]\to 1+\imath\R\subset\C^*$ satisfies 
\begin{equation}\label{eqtildeh=1}
	\tilde h=1\quad \text{in }P\times ([0,1]\setminus J')\supset P\times (J\cup\{0,1\})
\end{equation} 
and, in view of \eqref{eq:gamma(p)},
\[
	\big| \mathscr{P}\big(\tilde h(p,\cdot)f(p,\cdot)\big) - \alpha(p) \big| < \epsilon
	 \quad \text{for all }p\in P.
\]
Moreover, if $\mathscr{P}(f(p,\cdot))=\alpha(p)$ for all $p\in Q$, then $\tilde h(p,\cdot)=1$ for all $p\in Q$. Since $\tilde hf=f$ in $P\times J$ and the functions $g_1,\ldots,g_N$ are supported on $J$, the period dominating property of $w$ provides a continuous map $\zeta_\alpha:P\to U$ such that
\begin{equation}\label{eq:period-h}
	\mathscr{P}\big(w(\zeta_\alpha(p),\cdot) \tilde h(p,\cdot) f(p,\cdot)\big) = \alpha(p)\quad \text{for all }p\in P,
\end{equation}
and if $\mathscr{P}(f(p,\cdot))=\alpha(p)$ for all $p\in Q$, then $\zeta_\alpha(p)=0$ for all $p\in Q$. The continuous function $h:P\times [0,1]\to\C^*$ given by $h(p,\cdot)=w(\zeta_\alpha(p),\cdot) \tilde h(p,\cdot)$ satisfies the conclusion of the corollary. Indeed, conditions \eqref{eq:w(p,)} and \eqref{eqtildeh=1} ensure that $h(\cdot,0)=h(\cdot,1)=1$. By \eqref{eq:period-h}, we have
\[
	\int_0^1 h(p,s)f(p,s)\, ds = \mathscr{P}\big(w(\zeta_\alpha(p),\cdot) \tilde h(p,\cdot) f(p,\cdot)\big) = \alpha(p)\quad \text{for all }p\in P.
\]
Moreover, if $\alpha(p)=\int_0^1 f(p,s)\, ds$ $\big(=\mathscr{P}(f(p,\cdot))\big)$ for all $p\in Q$, then it is ensured that $\tilde h(p,\cdot)=1$ and $\zeta_\alpha(p)=0$ for all $p\in Q$, and hence $h(p,\cdot)=1$ for all $p\in Q$. Finally, \eqref{eq:w(p,)}, \eqref{eq:Rew-1}, \eqref{eqtildeh=1}, and the facts that $\zeta_\alpha(P)\subset U$ and $\Re(\tilde h)=1$ in $P\times[0,1]$ guarantee that $|\Re(h)-1|<\sigma$ in $P\times[0,1]$.
\end{proof}

%
%
\begin{proof}[Proof of Lemma \ref{lem:critical-case}]
Let $v$ denote the only critical point of $\rho$ in $L\setminus \mathring K$, and note that $v\in \mathring L\setminus K$. We distinguish cases depending on the Morse index of $\rho$ at $v$.

Assume first that the Morse index of $\rho$ at $v$ equals $0$. We then proceed as in the proof of \cite[Lemma 4.3]{AlarconForstnericLopez2019JGEA}. In this case a new connected component of the sublevel set $\{\rho\le s\}$ appears when $s$ passes the value $\rho(v)$, and hence a new connected and simply connected component $K'$ of $L$ appears. In particular, $K$ is a strong deformation retract of $L\setminus K'$. Thus, Lemma \ref{lem:noncritical-case} provides a continuous family of functions $\widetilde\varphi_p:L\setminus K'\to\C^*$ $(p\in[0,1]^k)$ of class $\mathscr{A}(L\setminus K')$ which is as close to the family $\varphi_p$ as desired in $[0,1]^k\times K$ and has $(\widetilde\varphi_p-\varphi_p)f_p\theta$ exact on $K$ for all $p\in [0,1]^k$. Moreover, we can choose $\widetilde\varphi_p=\varphi_p$ for all $p\in[0,1]^{k-1}\times\{0\}$ provided that $\varphi_p$ is of class $\mathscr{A}(L)$ and nowhere vanishing on $L$ for every $p\in[0,1]^{k-1}\times\{0\}$.  (Here we need the maps $f_p$ to be full.)  To finish it suffices to extend the family $\widetilde \varphi_p$ $(p\in [0,1]^k)$ to $K'$ as a continuous family of nowhere vanishing functions of class $\mathscr{A}(K')$, choosing $\widetilde \varphi_p|_{K'}=\varphi_p|_{K'}$ for all $p\in[0,1]^{k-1}\times\{0\}$ if for every $p\in[0,1]^{k-1}\times\{0\}$ the map $\varphi_p$ is of class $\mathscr{A}(L)$, vanishes nowhere on $L$, and satisfies $\int_C \varphi_pf_p\theta = F_p([C])$ for all closed curves $C\subset L$.

Assume that, on the contrary, the Morse index of $\rho$ at $v$ equals $1$. In this case there is a smooth Jordan arc $\gamma$ in $\mathring L\setminus \mathring K$, with its two endpoints in $bK$ and otherwise disjoint from $K$, such that $S=K\cup\gamma\subset\mathring L$ is an admissible subset of $M$ and a strong deformation retract of $L$. Denote by $\Omega$ the connected component of $\mathring L\setminus K$ intersecting $\gamma$. Note that $\Omega$ contains $\gamma$ except for its endpoints.

Fix $p\in [0,1]^k$. We claim that 
\begin{equation}\label{eq:f_pR-full}
	\text{the real span of $f_p(\Omega)$ equals $\C^n=\R^{2n}$}. 
\end{equation}
Indeed, since the map $f_p\colon M\to \C^n$ is holomorphic and full, $f_p|_\Omega:\Omega\to\C^n$ is full as well, and hence if $\lambda:\C^n\to\C$ is a $\C$-linear functional and $\lambda\circ f_p|_\Omega=0$, then $\lambda=0$.  Let $\mu:\C^n\to\R$ be an $\R$-linear functional with $\mu\circ f_p|_\Omega=0$.  Now, $\mu$ is the real part of a $\C$-linear functional $\lambda:\C^n\to\C$.  Since $f$ is holomorphic and $\Re\lambda\circ f_p|_\Omega=0$, we have $\lambda\circ f_p|_\Omega=0$, so $\lambda=0$ and hence $\mu=0$. This guarantees \eqref{eq:f_pR-full}, and thus there are points $x_1,\ldots,x_{2n}\in \Omega$ such that 
\[
	{\rm span}_\R\{f_p(x_1),\ldots,f_p(x_{2n})\}=\C^n.
\]
(Here, again, we need the maps $f_p$ to be full.)  Since $f_p$ depends continuously on $p\in [0,1]^k$, the same holds if we replace $p$ by any point $p'$ in a small neighbourhood of $p$ in $[0,1]^k$. Thus, compactness of $[0,1]^k$ ensures the existence of a finite set $\Lambda\subset \Omega\subset \mathring L\setminus K$ such that
\begin{equation}\label{eq:LambdaR-full}
	{\rm span}_\R\{f_p(x)\colon x\in\Lambda\}=\C^n
	\quad\text{for all }p\in [0,1]^k. 
\end{equation}
Since $\Omega$ is connected, we may deform the arc $\gamma$ outside a neighbourhood of its endpoints so that $\Lambda\subset\gamma$. Since $\Lambda\cap K=\varnothing$, we have that $\Lambda$ lies in the relative interior of $\gamma$, and since $\Lambda$ is finite there is a (closed) sub-arc $\gamma'$ in the relative interior of $\gamma$ such that $\Lambda\subset\gamma'$. Condition \eqref{eq:LambdaR-full} then ensures that the map $f_p$ is $\R$-full on $\gamma'$ for all $p\in [0,1]^k$. On the other hand, the fullness of the holomorphic map $f_p:M\to\C^n$ implies that $f_p$ is $\C$-full on every sub-arc of $\gamma$ for all $p\in [0,1]^k$. (Here, once again, we need the maps $f_p$ to be full.)  Corollary \ref{co:R-full-k} then enables us to extend the family of functions $\varphi_p:K\to\C^*$ to a continuous family of functions $\phi_p:S=K\cup\gamma\to\C^*$ $(p\in [0,1]^k)$ of class $\mathscr{A}(S)$ such that
\[
	\int_C \phi_pf_p\theta = F_p([C])\quad
	\text{for all closed curves $C\subset S$ and all }p\in[0,1]^k.
\]
Furthermore, if $\varphi_p$ is of class $\mathscr{A}(L)$, vanishes nowhere on $L$, and satisfies
\[
	\int_C \varphi_pf_p\theta = F_p([C])\quad
	\text{for all closed curves $C\subset L$, for all }p\in[0,1]^{k-1}\times\{0\},
\]
then we can choose $\phi_p=\varphi_p|_S$ for all $p\in [0,1]^{k-1}\times\{0\}$. Since $S$ is a strong deformation retract of $L$, this reduces the proof to the noncritical case granted in Lemma \ref{lem:noncritical-case}.
\end{proof}

Theorem \ref{th:main-technical-theorem} is thus proved.


\section{Proofs}
\label{sec:proofs}

\noindent
In this section, we prove the two main Theorems \ref{th:fibration} and \ref{th:inclusion}.

\begin{proof}[Proof of Theorem \ref{th:fibration}]
Fix an integer $k\ge 1$ and let $P=[0,1]^{k-1}$.  For $k=1$, we identify $P$ and $P\times\{0\}$ with $\{0\}\subset [0,1]$ and $P\times[0,1]$ with $[0,1]$.  Assume that we have continuous maps $u:P\times\{0\}\to \CMIf(M,\R^n)$, $G:P\times[0,1]\to \Of(M,\bold Q^{n-2})$, and $F:P\times[0,1]\to H^1(M,\R^n)$
such that 
\[
	G|_{P\times\{0\}} = \Gscr\circ u
	\quad\text{and}\quad
	F|_{P\times\{0\}} = \Flux\circ u;
\]
that is, the following square commutes.
\begin{equation}\label{eq:commuting-square} 
	\xymatrix{
	P\times\{0\} \ar@{^{(}->}[dd] \ar[rr]^{\!\!\!\!u} & & \CMIf(M,\R^n) \ar[dd]^{(\Gscr,\Flux)}
		\\ \\
	P\times[0,1] \ar[rr]^{ \!\!\!\!\!\!\!\!\!\!\!\!\!\!\!\!\!\!\!\!\!\!\!\!\!\!\!\!\!\!\!\!\!\! (G,F)} \ar@{-->}[uurr]^{\phi} & & \Of(M,\bold Q^{n-2})\times H^1(M,\R^n)
	} 
\end{equation}
To complete the proof it suffices to show that the map $(G,F)$ lifts to a continuous map $\phi:P\times[0,1]\to \CMIf(M,\R^n)$ such that
\begin{equation}\label{eq:phi}
	u = \phi|_{P\times\{0\}},\quad 
	G = \Gscr\circ \phi,
	\quad\text{and}\quad
	F = \Flux \circ \phi;
\end{equation}
that is, the two triangles in \eqref{eq:commuting-square} commute.  Write 
\[
	u(p,0)=u_p^0,\quad G(p,t)=G_p^t,
	\quad\text{and}\quad
	F(p,t)=F_p^t\quad \text{for all } (p,t)\in P\times [0,1].
\]
Since the square \eqref{eq:commuting-square} commutes, we have 
\begin{equation}\label{eq:Gp0-Fp0}
	\Gscr(u_p^0)=G_p^0 \quad\text{and}\quad
	\Flux(u_p^0)=F_p^0 \quad\text{for all } p\in P.
\end{equation}

Fix a holomorphic $1$-form $\theta$ vanishing nowhere on $M$. The Gauss map assignment $\Gscr:\CMIf(M,\R^n)\to \Of(M,\bold Q^{n-2})$ is naturally factored as
\[ 
\xymatrix{
\CMIf(M,\R^n) \ar[r]^{\,\,\,\psi} & \Of(M, \boldA_*) \ar[r]^{\!\!\!\!\pi_*} & \Of(M,\bold Q^{n-2}),
} 
\]
where $\psi(u)=2\partial u/\theta$ and $\pi:\boldA_*=\{z\in \C_*^n\colon z_1^2+\cdots+z_n^2=0\}\to \bold Q^{n-2}$ is the restriction of the canonical projection $\pi:\C^n_*=\C^n\setminus\{0\}\to\C\P^{n-1}$. Here, $\Of(M, \boldA_*)$ denotes the space of all holomorphic maps $f:M\to\boldA_*\subset\C^n$ that are {\em full\,} in the sense that the $\C$-linear span of $f(M)$ is all of $\C^n$. Recall that every holomorphic map $g\colon M\to \C\P^{n-1}$ lifts to a holomorphic map $f:M\to\C^n_*$ such that $g=\pi\circ f=\pi_*(f)$; moreover, $g$ is full if and only if $f$ is full, while $g(M)\subset \bold Q^{n-2}$ if and only if $f(M)\subset \boldA_*$.  Also recall that a conformal minimal immersion $u\in\CMI(M,\R^n)$ is recovered from $\psi(u)$ by the integral formula
\[
	u(x)=u(x_0)+\Re\int_{x_0}^x\psi(u) \theta,\quad x\in M,
\]
for any base point $x_0\in M$, while 
\[
	\Flux(u)([C])=\Im\int_C\psi(u)\theta,\quad [C]\in H_1(M,\Z).
\]

Since $\Gscr=\pi_*\circ\psi$, the first condition in \eqref{eq:Gp0-Fp0} ensures that the following square of continuous maps commutes.
\begin{equation}\label{eq:commuting-square-2} 
	\xymatrix{
	P\times\{0\} \ar@{^{(}->}[dd] \ar[rr]^{\psi\circ u} & & \Of(M, \boldA_*) \ar[dd]^{\pi_*}
		\\ \\
	P\times[0,1] \ar[rr]^{\!\!\!\!\!\!\!\!\!\!\! G} \ar@{-->}[uurr]^{\varrho} & & \Of(M,\bold Q^{n-2})
	} 
\end{equation}

Since $\pi_*: \Of(M, \boldA_*) \to \Of(M,\bold Q^{n-2})$ is a fibration (see \cite[Lemma 5.1]{AlarconForstnericLopez2019JGEA}), the map $G$ lifts to a continuous map $\varrho:P\times[0,1]\to  \Of(M, \boldA_*)$ such that the two triangles in \eqref{eq:commuting-square-2}  commute:
\begin{equation}\label{eq:varrho}
	\psi\circ u = \varrho|_{P\times\{0\}}
	\quad\text{and}\quad
	G = \pi_* \circ \varrho.
\end{equation}
Write $\varrho(p,t)=\varrho_p^t$ for all $(p,t)\in P\times [0,1]$. Theorem \ref{th:main-technical-theorem} provides a continuous family of holomorphic functions $\varphi_p^t:M\to\C^*$, $(p,t)\in P\times[0,1]$, such that
\begin{equation}\label{eq:taupt}
	\int_C \varphi_p^t \varrho_p^t\theta = \imath F_p^t([C])\in\imath\R^n
	\quad \text{for all closed curves $C\subset M$, $(p,t)\in P\times[0,1]$}.
\end{equation}
Furthermore, since $\varrho_p^0=2\partial u_p^0/\theta$ by the first condition in \eqref{eq:varrho} and $F_p^0=\Flux(u_p^0)$ by the second condition in \eqref{eq:Gp0-Fp0} for all $p\in P$, we have that 
\[
	\int_C \varrho_p^0\theta = \imath F_p^0([C])
	\quad \text{for all closed curves $C\subset M$ and all $p\in P$},
\]
so we can choose the family of maps $\varphi_p^t$ so that 
\begin{equation}\label{eq:taup0}
	\varphi_p^0=1\quad \text{for all }p\in P.
\end{equation}
Fix a point $x_0\in M$. It follows from \eqref{eq:taupt} that the map $\phi_p^t:M\to\R^n$ defined by
\[
	\phi_p^t(x)=u_p^0(x_0)+\Re\int_{x_0}^x \varphi_p^t \varrho_p^t\theta,\quad x\in M,
\]
is a well-defined full conformal minimal immersion (note that $\varphi_p^t \varrho_p^t:M\to \boldA_*\subset\C^n$ is holomorphic and full, and $\Re(\varphi_p^t \varrho_p^t\theta)$ is exact). We claim that the continuous map $\phi:P\times[0,1]\to \CMIf(M,\R^n)$ given by $\phi(p,t)=\phi_p^t$ for all $(p,t)\in P\times[0,1]$ satisfies the conditions in \eqref{eq:phi}. Indeed, in view of \eqref{eq:taup0} and the first condition in \eqref{eq:varrho}, we have
\[
	\phi_p^0(x) = u_p^0(x)+\Re\int_{x_0}^x\varrho_p^0\theta =
	u_p^0(x)+\Re\int_{x_0}^x 2\partial u_p^0 = u_p^0(x),\quad x\in M,\; p\in P,
\]
so $u=\phi|_{P\times\{0\}}$.  On the other hand,
\[
	\Gscr(\phi_p^t) = \pi_*(2\partial \phi_p^t/\theta) = \pi_*(\varphi_p^t\varrho_p^t) = \pi_*(\varrho_p^t) = G_p^t,\quad (p,t)\in P\times[0,1],
\]
where we have used that $\Gscr=\pi_*\circ\psi$, the fact that $\varphi_p^t$ takes values in $\C^*$, and the second condition in \eqref{eq:varrho}. Therefore, $G = \Gscr\circ \phi$. Finally, \eqref{eq:taupt} directly implies that $\Flux(\phi_p^t)=F_p^t$ for all $(p,t)\in P\times[0,1]$, that is, $F = \Flux \circ \phi$.
\end{proof}

We now turn to the proof of Theorem \ref{th:inclusion}.  First we need the following h-principle, which easily implies Corollary \ref{co:dense}.

%
%
\begin{theorem}\label{th:h-principle}
Let $M$ be an open Riemann surface, $n\geq 3$, $Q$ be a closed subset of a contractible finite CW-complex $P$, and $G\colon M\times P\to \bold Q^{n-2}$ be a continuous map such that $G_p:=G(\cdot,p)\in \Of(M,\bold Q^{n-2})$ for all $p\in P$.  For any $\Oscr(M)$-convex compact set $K\subset M$ and any $\epsilon>0$ there is a homotopy $G^t\colon M\times P\to \bold Q^{n-2}$, $t\in[0,1]$, satisfying the following conditions.
\begin{enumerate}[\rm (i)]
\item  $G_p^t:=G^t(\cdot,p)\colon M\to \bold Q^{n-2}$ lies in $\Of(M, \bold Q^{n-2})$ for all $(p,t)\in P\times[0,1]$.
\smallskip
\item  $G_p^t=G_p$ for all $(p,t)\in (P\times\{0\})\cup(Q\times[0,1])$.
\smallskip
\item $|G_p^t(x)-G_p(x)|<\epsilon$ for all $x\in K$ and $(p,t)\in P\times[0,1]$.
\smallskip
\item  $G_p^t\in \Ofc(M, \bold Q^{n-2})$ for all $(p,t)\in (P\setminus Q)\times (0,1]$.
\end{enumerate}
In particular, if in addition $G_p\in \Ofc(M,\bold Q^{n-2})$ for all $p\in Q$, then we have $G_p^t\in \Ofc(M, \bold Q^{n-2})$ for all $(p,t)\in P\times (0,1]$.
\end{theorem}

%
%
\begin{proof}
Since $\Gscr:\CMIf(M, \R^n) \to \Of(M,\bold Q^{n-2})$ is a fibration by Theorem \ref{th:fibration} and $P$ is contractible, we can lift the given map $G\colon M\times P\to \bold Q^{n-2}$ by $\Gscr$ to a map $u\colon M\times P\to\R^n$ such that $u_p=u(\cdot,p) \in\CMIf(M,\R^n)$ for all $p\in P$. By the parametric h-principle \cite[Theorem 6.1(a)]{AlarconLarusson2021}, for any $\delta>0$, there is a homotopy $u^t\colon M\times P\to \R^n$, $t\in[0,1]$, with the following properties.
\begin{enumerate}[\rm (a)]
\item  $u_p^t:=u^t(\cdot,p)\colon M\to \R^n$ lies in $\CMIf(M, \R^n)$ for all $(p,t)\in P\times[0,1]$.
\smallskip
\item  $u_p^t=u_p$ for all $(p,t)\in (P\times\{0\})\cup(Q\times[0,1])$.
\smallskip
\item  $|u_p^t(x)-u_p(x)|<\delta$ for all $x\in K$ and $(p,t)\in P\times[0,1]$.
\smallskip
\item  $u_p^t\in \CMIfc(M, \R^n)$ for all $p\in (P\setminus Q)\times (0,1]$.
\end{enumerate}
Setting $G_p^t=\Gscr(u_p^t)$ with $\delta>0$ small enough defines a homotopy as desired.  Note in particular that if $G_p\in \Ofc(M,\bold Q^{n-2})$ for all $p\in Q$, then we cannot assert that $u_p$ is complete for $p\in Q$, but we have $\Gscr(u_p) = G_p\in\Ofc(M, \bold Q^{n-2})$ for all $p\in Q$, and hence {\rm (b)} and {\rm (d)} ensure that $G_p^t\in \Ofc(M, \bold Q^{n-2})$ for all $(p,t)\in P\times (0,1]$.
\end{proof}

\begin{proof}[Proof of Theorem \ref{th:inclusion}]
(a)  Applying Theorem \ref{th:h-principle} with $P$ a singleton and $Q$ empty shows that the inclusion $j:\Ofc(M,\bold Q^{n-2}) \hookrightarrow \Of(M,\bold Q^{n-2})$ induces a surjection of path components.  Applying Theorem \ref{th:h-principle} with $P$ a closed ball of dimension $k\geq 1$ and $Q$ the boundary sphere of $P$ shows that $j$ induces a monomorphism at the level of $\pi_{k-1}$ and an epimorphism at the level of $\pi_k$.

(b)  It suffices to show that the spaces $\Ofc(M,\bold Q^{n-2})$ and $\Of(M,\bold Q^{n-2})$ are ANRs.  First, $\Of(M,\bold Q^{n-2})$ is an ANR, being open in $\Oscr(M,\bold Q^{n-2})$, which is ANR by \cite[Theorem 9]{Larusson2015PAMS} since $\bold Q^{n-2}$ is an Oka manifold.  Second, since $\Of(M,\bold Q^{n-2})$ is an ANR, Theorem \ref{th:h-principle} and \cite[Proposition 5.2]{AlarconLarusson2021} imply that $\Ofc(M,\bold Q^{n-2})$ is an ANR.

(c)  It is not difficult to adapt the general position theorem \cite[Theorem 5.4]{ForstnericLarusson2019CAG} for maps into $\bold A_*$ to maps into $\bold Q^{n-2}$ using the fact that if $P$ is a contractible finite CW-complex, then every continuous map $P\to\Oscr(M, \bold Q^{n-2})$ lifts by the projection $\pi:\boldA_*\to\bold Q^{n-2}$, whose fibre $\C^*$ is Oka, to a continuous map $P\to\Oscr (M,\boldA_*)$, and conclude that the inclusion $\Of(M,\bold Q^{n-2})\hookrightarrow \Oscr(M,\bold Q^{n-2})$ is a weak homotopy equivalence.\footnote{For the hyperquadric $\boldA$ (as opposed to a more general cone $A$ as in \cite[Theorem 5.4]{ForstnericLarusson2019CAG}), the notion of nondegeneracy in \cite {AlarconForstneric2014IM} and \cite{ForstnericLarusson2019CAG} is equivalent to nonflatness by \cite[Lemma 2.3]{AlarconForstnericLopez2016MZ}.  To adapt \cite[Theorem 5.4]{ForstnericLarusson2019CAG} to full maps in place of nonflat maps, in its proof, simply invoke the proof of \cite[Theorem 3.1(a)]{AlarconForstnericLopez2016MZ} instead of the proof of \cite[Theorem 2.3(a)]{AlarconForstneric2014IM} (the latter theorem is incorrectly referred to as Theorem 3.2(a) in \cite{ForstnericLarusson2019CAG}).  Beware that fullness is called nondegeneracy in \cite{AlarconForstnericLopez2016MZ}.}  By the basic Oka principle, the inclusion $\Oscr(M,\bold Q^{n-2})\hookrightarrow \Cscr(M,\bold Q^{n-2})$ is also a weak homotopy equivalence.  Finally, if $M$ has finite topological type, then $\Cscr(M,\bold Q^{n-2})$ is an ANR (see \cite[Proposition 7]{Larusson2015PAMS} and the references in its proof).
\end{proof}


\section{The fibre of the Gauss map assignment}
\label{sec:fibre}

\noindent
As before, we let $M$ be an open Riemann surface and $n\geq 3$.  By Theorem \ref{th:fibration}, the Gauss map assignment $\Gscr: \CMIf(M,\R^n) \to \Of(M,\bold Q^{n-2})$ is a Serre fibration.  As shown already, $\Of(M,\bold Q^{n-2})$ has the weak homotopy type of $\Cscr(M,\bold Q^{n-2})$.  Moreover, $\bold Q^{n-2}$ is simply connected: for $n=3$, $\bold Q^{n-2}$ is isomorphic to the Riemann sphere; for $n\geq 4$, we invoke the fact that a smooth hypersurface in $\C\P^{n-1}$ is simply connected.  It follows that $\Of(M,\bold Q^{n-2})$ is path connected, so the fibres of $\Gscr: \CMIf(M,\R^n) \to \Of(M,\bold Q^{n-2})$ all have the same weak homotopy type.  In this section, we shall determine this homotopy type.  

Recall the factorisation 
\[ \xymatrix{
\CMIf(M,\R^n) \ar[r]^{\,\,\,\psi} & \Of(M, \boldA_*) \ar[r]^{\!\!\!\!\pi_*} & \Of(M,\bold Q^{n-2}),
} \]
of $\Gscr$ used in Section \ref{sec:proofs} in the proof of Theorem \ref{th:fibration}, where $\psi(u)=2\partial u/\theta$, with $\theta$ being a nowhere-vanishing holomorphic 1-form on $M$, and $\pi_*$ is induced by the projection $\pi$ from the punctured null quadric $\bold A_*$ in $\C_*^n$ onto the hyperquadric $\bold Q^{n-2}$ in $\C\P^{n-1}$.  The projection $\pi:\boldA_* \to \bold Q^{n-2}$ is a fibre bundle with fibre $\C^*$.  Note that $\Gscr$ and $\pi_*$ are canonically defined, whereas $\psi$ depends on the choice of $\theta$.  

In the proof of Theorem \ref{th:fibration}, we used the result that $\pi_*$ is a fibration \cite[Lemma 5.1]{AlarconForstnericLopez2019JGEA}.  In fact, the lemma shows that $\pi_*:\Oscr(M, \boldA_*) \to \Oscr(M,\bold Q^{n-2})$ is a fibration.  Its fibre $F_0$, well defined up to weak homotopy equivalence, is $\Oscr(M, \C^*)$ or, by the basic Oka principle, $\Cscr(M, \C^*)$.  Hence, $F_0$ has the weak homotopy type of a countably infinite disjoint union of circles, unless $M$ is the plane or the disc, in which case $F_0$ has the weak homotopy type of a circle.  We also need the result that $\psi$ is a weak homotopy equivalence.  This follows, by an argument similar to the proof of \cite[Theorem 5.6]{ForstnericLarusson2019CAG}, from the parametric h-principle \cite[Theorem 5.3]{ForstnericLarusson2019CAG} (the version with vanishing {\em real\,} periods) adapted to full maps in place of nonflat maps.

Let $F$ be the fibre of $\Gscr$ and consider the following commuting diagram.
\[ \xymatrix{
F \ \ar@{^{(}->}[r] \ar[d]^\psi & \CMIf(M,\R^n) \ar[r]^\Gscr \ar[d]^\psi & \Of(M,\bold Q^{n-2}) \ar@{=}[d] \\
F_0 \ \ar@{^{(}->}[r] &  \Of(M, \boldA_*) \ar[r]^{\!\!\!\!\pi_*} & \Of(M,\bold Q^{n-2})
} \]
The associated long exact sequences of homotopy groups show that $\psi$ induces a weak homotopy equivalence $F\to F_0$, so Theorem \ref{th:fibre} is proved.


\subsection*{Acknowledgements}
A.\ Alarc\'on was partially supported by the State Research Agency (AEI) 
via the grant no.\ PID2020-117868GB-I00, and the ``Maria de Maeztu'' Excellence Unit IMAG, reference CEX2020-001105-M, funded by MCIN/AEI/10.13039/501100011033/; and the Junta de Andaluc\'ia grant no. P18-FR-4049; Spain.


{\bibliographystyle{abbrv} \bibliography{references}}


\newpage
\noindent Antonio Alarc\'{o}n

\noindent Departamento de Geometr\'{\i}a y Topolog\'{\i}a e Instituto de Matem\'aticas (IMAG), Universidad de Granada, Campus de Fuentenueva s/n, E--18071 Granada, Spain

\noindent  e-mail: {\tt alarcon@ugr.es}

\bigskip
\noindent Finnur L\'arusson

\noindent School of Mathematical Sciences, University of Adelaide, Adelaide SA 5005, Australia

\noindent  e-mail: {\tt finnur.larusson@adelaide.edu.au}

\end{document}